\documentclass[12pt]{amsart}
\usepackage{amssymb, amsmath, mathtools}
\usepackage[utf8]{inputenc}
\usepackage[T1]{fontenc} 
\usepackage{lmodern} 
\usepackage[colorlinks=true, citecolor=blue, linkcolor=red, urlcolor=blue, pagebackref, hyperindex]{hyperref}
\usepackage{fullpage}

\usepackage{amsmath,amscd,mathrsfs}
\usepackage{enumitem}

\usepackage{tikz}
\usepackage[all]{xy}  
\newcommand{\lr}[1]{\langle {#1} \rangle}

\newtheorem{Theoremx}{Theorem}
\newtheorem{Theorem}{Theorem}[section]
\newtheorem{Potential Theorem}[Theorem]{Potential Theorem}
\newtheorem{Lemma}[Theorem]{Lemma}

\newtheorem{Proposition}[Theorem]{Proposition}

\theoremstyle{definition}

\newtheorem{Definition}[Theorem]{Definition}

\newtheorem{Examples}[Theorem]{Examples}

\theoremstyle{remark} 
\newtheorem{Remark}[Theorem]{Remark}

\DeclareMathOperator{\height}{ht}

\DeclareMathOperator{\Max}{Max}
\DeclareMathOperator{\Sym}{Sym}

\DeclareMathOperator{\Hom}{Hom}
\DeclareMathOperator{\Spec}{Spec}
\DeclareMathOperator{\Min}{Min}

\DeclareMathOperator{\rank}{rank}

\def\p{\mathfrak{p}}
\def\q{\mathfrak{q}}
\def\m{\mathfrak{m}}

\def\a{\mathfrak{a}}

\def\C{\mathscr{C}}

\def\E{\mathscr{E}}
\def\N{\mathbb{N}}

\def\ds{\displaystyle}

\newcommand{\gs}{\geqslant}
\newcommand{\ls}{\leqslant}

\begin{document}

\title{Generalizing Serre's Splitting Theorem and Bass's Cancellation Theorem via free-basic elements}
\author{Alessandro De Stefani}
\address{Department of Mathematics, Royal Institute of Technology (KTH), Stockholm, 100 44, Sweden}
\email{ads@kth.se}
\author{Thomas Polstra}
\address{Department of Mathematics, University of Missouri-Columbia, Columbia, MO 65211}
\email{tmpxv3@mail.missouri.edu}
\author{Yongwei Yao}
\address{Department of Mathematics and Statistics, Georgia State University, 30 Pryor Street, Atlanta, GA 30303}
\email{yyao@gsu.edu}

\maketitle
\begin{abstract} We give new proofs of two results of Stafford from \cite{Stafford}, which generalize two famous Theorems of Serre and Bass regarding projective modules. Our techniques are inspired by the theory of basic elements. Using these methods we further generalize Serre's Splitting Theorem by imposing a condition to the splitting maps, which has an application to the case of Cartier algebras.
\end{abstract} 

\section{Introduction}

Throughout this section, $R$ denotes a commutative Noetherian ring with unity. The existence of unimodular elements has always played a crucial role in algebraic $K$-theory.  An element of a projective module is called unimodular if it generates a free summand. Note that unimodular elements are basic elements, a notion introduced by Swan \cite{Swan}: an element of an module is called basic if remains a minimal generator after localization at every prime of the ring. In a fundamental paper from 1971, Eisenbud and Evans make use of basic elements for modules that are not necessarily projective. Using basic element theory, Eisenbud and Evans generalize and greatly improve several theorems about projective modules, as well as about general finitely generated modules \cite{EisenbudEvans}. Two such theorems are Serre's Splitting Theorem and Bass's Cancellation Theorem on projective modules.

\begin{Theorem}[\cite{Serre}, Theorem 1 and \cite{Bass}, Theorem 8.2]\label{Serre's Theorem on Projective Modules} Let $R$ be a commutative Noetherian ring of Krull dimension $d$. Let $P$ be a finitely generated projective $R$-module whose rank at each localization at a prime ideal is at least $d+1$. Then $P$ contains a free $R$-summand. 
\end{Theorem}

\begin{Theorem}[\cite{Bass}, Theorem 9.1] \label{Bass's Cancellation Theorem} Let $R$ be a commutative Noetherian ring of finite Krull dimension $d$, and $P$ be a finitely generated projective $R$-module whose rank at each localization at a prime ideal is at least $d+1$. Let $Q$ be a finitely generated projective $R$-module, and assume that $Q \oplus P \cong Q \oplus N$ for some $R$-module $N$. Then $P \cong N$.
\end{Theorem}

Given a projective module $P$, an equivalent way of saying that $P_\p$ has rank at least $d+1$ for all $\p \in \Spec(R)$ is to say that, locally at each prime, it contains a free summand of rank at least $d+1$. In this article, we present new methods to extend Theorems~\ref{Serre's Theorem on Projective Modules} and~\ref{Bass's Cancellation Theorem} to all finitely generated modules, not necessarily projective, that contain a free summand of large enough rank at each localization at a prime.

\begin{Theoremx} \label{Serre's Theorem Intro} Let $R$ be a commutative Noetherian ring, and $M$ a finitely generated $R$-module. Assume that $M_\p$ contains a free $R_\p$ summand of rank at least $\dim (R/\p)+1$ for each $\p\in \Spec(R)$. Then $M$ contains a free $R$-summand.
\end{Theoremx}

\begin{Theoremx} \label{Bass's Theorem Intro} Let $R$ be a commutative Noetherian ring and $M$ a finitely generated $R$-module such that, for each $\p\in \Spec(R)$, $M_\p$ contains a free $R_\p$ summand of rank at least $\dim(R/\p)+1$. For any finitely generated projective $R$-module $Q$ and any finitely generated $R$-module $N$, if $Q\oplus M \cong Q \oplus N$, then $M \cong N$.
\end{Theoremx}

Stafford proved analogue results in \cite[Corollaries 5.9 and 5.11]{Stafford} where, instead of looking at local number of free summands, he used the notion of $r$-$\rank$. Stafford's results are actually very general in that he proved analogues of Theorem~\ref{Serre's Theorem Intro} and Theorem~\ref{Bass's Theorem Intro} for rings that are not necessarily commutative.

Let $M$ be a finitely generated $R$-module. Observe that the conclusion of Theorem~\ref{Serre's Theorem Intro} can be viewed as a statement about the existence of a surjective homomorphism inside $\Hom_R(M,R)$. For several applications, it is useful to restrict the selection of homomorphisms $M \to R$ to those belonging to a given $R$-submodule $\E$ of $\Hom_R(M,R)$. This is the scenario arising, for instance, from the study of the F-signature of Cartier subalgebras of $\C^R = \oplus_e \Hom_R(F^e_*R,R)$, where $R$ is an F-finite local ring of prime characteristic \cite{BST2012}. We can naturally generalize the definitions in this paper and results about free-basic elements to the setup of $R$-submodules $\E$ of $\Hom_R(M,R)$. The main achievement in this direction is a new generalization of Theorem~\ref{Serre's Theorem Intro}.

\begin{Theoremx}\label{Generalized Serre 2 Intro} 
Let $R$ be a commutative Noetherian ring, $M$ a finitely generated $R$-module, and $\E$ an $R$-submodule of $\Hom_R(M,R)$. Assume that, for each $\p \in \Spec(R)$, $M_\p$ contains a free $\E_\p$-summand of rank at least $\dim(R/\p)+1$.  Then $M$ contains a free $\E$-summand.
\end{Theoremx}

As the main application of Theorem~\ref{Generalized Serre 2 Intro}, the authors establish in \cite{DSPY} the existence of a global F-signature with respect to Cartier subalgebras of $\C^R$, where $R$ is an F-finite ring of prime characteristic, not necessarily local. 

The statements of Theorems~\ref{Serre's Theorem Intro},~\ref{Bass's Theorem Intro} and~\ref{Generalized Serre 2 Intro} can actually be strengthen by just requiring that $M_\p$ has a free summand (free $\E_\p$-summand for Theorem~\ref{Generalized Serre 2 Intro}) of rank at least $1+\dim_{\tiny{j\mbox{-}\Spec(R)}}(\p)$, for each $\p \in j$-$\Spec(R)$. Here, $\dim_{j\mbox{-}\Spec(R)}(\p)$ denotes the supremum over the lengths of increasing chains of primes in $j$-$\Spec(R)$ that start with $\p$. In addition, in Theorem~\ref{Serre's Theorem Intro}  (respectively, Theorem~\ref{Generalized Serre 2 Intro}), we can get even stronger conclusions, provided some additional assumptions hold. For instance, assume that there exists a fixed $R$-submodule $N$ of $M$ such that, for all $\p \in \Spec(R)$, $M_\p$ has a free $R_\p$-summand (respectively, free $\E_\p$-summand) of rank at least $\dim(R/\p)+1$ that is contained in $N_\p$. Then, the global free $R$-summand of $M$ can be realized inside $N$ as well. Furthermore, if for all $\p \in \Spec(R)$ the rank of the free summand of $M_\p$ is at least $\dim(R/\p)+i$ for some fixed positive integer $i$, then the global free summand of $M$ can be realized of rank at least $i$.

This article is structured as follows: in Section~\ref{SectionElemApproach}, we present a version of Theorem~\ref{Serre's Theorem Intro} via rather direct and elementary techniques. However, the methods we employ are not effective enough to prove the result in its full generality. In Section~\ref{Free-Basic Section}, we introduce the notion of free-basic element, which allows us to prove Theorem~\ref{Serre's Theorem Intro} and Theorem~\ref{Bass's Theorem Intro} in their full generality. We believe that free-basic elements are worth exploring and interesting on their own, as they share many good properties both with basic elements and unimodular elements. In Section~\ref{Cartier Section}, we generalize free-basic elements to the setup of $R$-submodules of $\Hom_R(M,R)$, and prove Theorem~\ref{Generalized Serre 2 Intro}.

\section{An elementary approach to Theorem~\ref{Serre's Theorem Intro}} \label{SectionElemApproach} 
Throughout this section, $R$ is a commutative Noetherian ring with identity, and of finite Krull dimension $d$. Given an ideal $I$ of $R$, we denote by $V(I)$ the closed set of $\Spec(R)$ consisting of all prime ideals that contain $I$. Given $s \in R$, we denote by $D(s)$ the open set consisting of all prime ideals that do not contain $s$. Recall that $\{D(s) \mid s \in R\}$ is a basis for the Zariski topology on $\Spec(R)$, and that $D(s)$ can be identified with $\Spec(R_s)$. We present an elementary approach to Theorem~\ref{Serre's Theorem Intro}, which will lead to a proof in the case when $R$ has infinite residue fields.  We first need to recall some facts about the symmetric algebra of a module. If $M$ is a non-zero finitely generated $R$-module, we denote by $\Sym_R(M)$ the symmetric algebra of $M$ over $R$. Given a presentation 
\[
\xymatrixcolsep{5mm}
\xymatrixrowsep{2mm}
\xymatrix{
R^{n_2}\ar[rr]^-{(a_{ij})} && R^{n_1}\ar[rr] && M\ar[rr] && 0
}
\]
then we can describe $\Sym_R(M)$, as an $R$-algebra, in the following way:
\[
\ds \Sym_R(M)\cong R[x_1,x_2,\ldots,x_{n_1}]/J, \ \ \mbox{ where } J=\left(\sum_{i=1}^{n_1}a_{ij}x_i\mid 1\ls j\ls n_2\right). 
\]
When $M$ is free of rank $n$, then $\Sym_R(M)$ is just a polynomial ring over $R$, in $n$ indeterminates. Huneke and Rossi give a formula for the Krull dimension of a symmetric algebra.

\begin{Theorem}[\cite{HunekeRossi}, Theorem 2.6 (ii)]\label{Huneke Rossi} Let $R$ be a ring which is the homomorphic image of a Noetherian universally catenary domain of finite dimension. Let $M$ be a finitely generated $R$-module, then $\ds\dim(\Sym_R(M))=\max_{\p\in\Spec(R)}\{\dim(R/\p)+\mu(M_\p)\}$.
\end{Theorem}

Now, suppose that $P$ is a finitely generated projective $R$-module, and suppose we have a surjection $R^n\to P \to 0$. Let $Q$ be its kernel. Applying $(-)^*:=\Hom_R(-, R)$ to the split short exact sequence $0\to Q\to R^n\to P \to 0$ gives another short exact sequence 
\[
\xymatrixcolsep{5mm}
\xymatrixrowsep{2mm}
\xymatrix{
0\ar[rr] &&  P^* \ar[rr] && (R^n)^*\cong R^n \ar[rr] && Q^* \ar[rr] && 0.
}
\]
This provides a natural way to view $\Sym_R(Q^*)$ as the homomorphic image of $R[x_1,\ldots,x_n] \cong \Sym_R(R^n)$. Moreover, as $Q$ is projective, if $\m\in \Max(R)$ is a maximal ideal, then $(Q/\m Q)^*\cong Q^*\otimes_R R/\m$, where the first $(-)^*$ is over $R/\m$ while the second is over $R$. In addition, we have $\Sym_{R/\m}(Q^*\otimes_R R/\m)\cong \Sym_R(Q^*)\otimes_R R/\m$. In fact, for any finitely generated $R$-module $M$, we have $\Sym_{R/\m}(M\otimes_R R/\m)\cong \Sym_R(M)\otimes_R R/\m$. To see this, one only needs to notice that if $R^{n_2}\rightarrow R^{n_1}\rightarrow M\rightarrow 0$ is a presentation of $M$ as an $R$-module, then applying $-\otimes_R R/\m$ gives a presentation of $M\otimes_R R/\m$ as an $R/\m$-vector space. The following lemma shows how the symmetric algebra can detect whether an element of a projective module is locally a minimal generator or not.

\begin{Lemma}\label{Lemma Max(Sym) VS MinGens} Let $R$ be a commutative Noetherian ring, $P$ be a finitely generated projective $R$-module with generators $\eta_1,\ldots,\eta_n$, and $Q$ be the kernel of the surjective map $R^n\rightarrow P$ which sends $e_i\mapsto \eta_i$. For any $r_1,\ldots,r_n\in R$ and $\m\in \Spec(R)$, we have $r_1\eta_1+\cdots+ r_n\eta_n\in \m P_\m$ if and only if $(\m,x_1-r_1,\ldots,x_n-r_n)$ is a maximal ideal of $\Sym_{R}(Q^*)$.
\end{Lemma}

\begin{proof} Checking that $r_1\eta_1+\cdots+ r_n\eta_n\in \m P_\m$ is equivalent to checking that $r_1\eta_1+\cdots+ r_n\eta_n=0$ in $P_\m/\m P_\m\cong P/\m P$. Let $\kappa := R/\m$. As $P$ is projective, tensoring with $\kappa$ gives a short exact sequence $0\rightarrow Q/\m Q \rightarrow \kappa^n\rightarrow P/\m P \rightarrow 0$ of $\kappa$-vector spaces. Let $t=\dim_{\kappa}(P/\m P)$, and choose a basis for $P /\m P \cong \kappa^t$. Let $(a_{ij})$ be the $t \times n$ matrix representing the onto map $\kappa^n\rightarrow P/\m P$. After dualizing, we see that $\Sym_{\kappa}((Q/\m Q)^*)\cong \kappa[x_1,\ldots,x_n]/J$ where $J=(\sum_{i=1}^na_{ji}x_i\mid 1\ls j\ls t)$. Furthermore, observe that the maximal ideal $(x_1-r_1,\ldots,x_n-r_n) \in \Max(\kappa[x_1,\ldots,x_n])$ contains $J$ if and only if $\sum_{i=1}^na_{ji}r_i=0$ for each $1\ls j\ls t$. This happens if and only if $(r_1,\ldots,r_n)(a_{ij})^T=0$, which in turn is equivalent to $(a_{ij})(r_1,\ldots,r_n)^T=0$. Finally, since $(a_{ij})$ is the matrix representing $\kappa^n \to P/\m P$, this happens if and only if $r_1\eta_1+\cdots +r_n\eta_n=0$ in $P/\m P$. Therefore, $r_1 \eta_1 + \ldots + r_n \eta_n \in \m P_\m$ if and only if $(x_1-r_1,\ldots,x_n-r_n) \in \Max(\Sym_{\kappa}(Q/\m Q)^*) = \Max(\Sym_R(Q^*) \otimes R/\m)$, if and only if $(\m,x_1-r_1,\ldots,x_n-r_n) \in \Max(\Sym_{R}(Q^*))$.
\end{proof}

Using Lemma~\ref{Lemma Max(Sym) VS MinGens}, we can give a proof of Serre's Splitting Theorem~\ref{Serre's Theorem on Projective Modules} \cite[Theorem ~1]{Serre} for algebras of finite type over an algebraically closed field. 
The main purpose of showing this argument here is for the reader to get more familiar with the techniques that will be employed and generalized later in this section. We thank Mohan Kumar for suggesting this proof to us, and for pointing out how to generalize it to any finitely generated module in the case when $R$ is an affine algebra over an algebraically closed field.

\begin{Theorem} \label{Serre alg closed} Let $R$ be a $d$-dimensional ring of finite type over an algebraically closed field $k$, and $P$ be a finitely generated projective $R$-module whose rank at each localization at a prime ideal is at least $d+1$. Then $P$ contains a free $R$-summand.
\end{Theorem}

\begin{proof} Let $\eta_1,\ldots,\eta_n$ be generators of $P$ and $Q$ be the kernel of the natural surjection $R^n\rightarrow P$. Then $\ds\dim(\Sym_R(Q^*))=\max_{\p\in\Spec(R)}\{\dim(R/\p)+\mu((Q^*)_\p)\}$ by Theorem~\ref{Huneke Rossi}. Observe that $Q$ is a locally free $R$-module of rank equal to $n-\rank(P)$. Therefore $Q^*$ is locally free of the same rank, and thus $\dim(\Sym_R(Q^*))\ls d+n-d-1=n-1$. Let $I$ be the ideal of $k[x_1,\ldots,x_n]$ defining the kernel of the composition $k[x_1,\ldots,x_n]\subseteq R[x_1,\ldots,x_n]\rightarrow \Sym_R(Q^*)$. As $\dim(k[x_1,\ldots,x_n])=n>\dim(\Sym_R(Q^*))$ it must be the case that $I\not=0$ and, by the Nullstellensatz, there must be a maximal ideal $(x_1-r_1,\ldots,x_n-r_n)$ of $k[x_1,\ldots,x_n]$ not containing $I$. In particular, for each $\m\in \Max(R)$, the ideal $(\m,x_1-r_1,\ldots,x_n-r_n)$ cannot be a maximal ideal of $\Sym_R(Q^*)$. By Lemma~\ref{Lemma Max(Sym) VS MinGens}, $r_1\eta_1+\ldots+r_n\eta_n$ is then a minimal generator of $P_\m$ for each $\m\in \Max(R)$. Therefore, the map $R\rightarrow P$ sending $1\mapsto r_1\eta_1+\cdots+r_n\eta_n$ splits locally at each $\m \in \Max(R)$, and hence it splits globally.  
\end{proof}

We now present a generalization of the techniques just employed. This allows us to prove Theorem~\ref{Serre's Theorem Intro} for rings $R$ of which all the residue fields $R/\m$, for $\m \in \Max(R)$, are infinite. This includes rings that contain an infinite field. We first prove an auxiliary Lemma.

\begin{Lemma} \label{Lemma increase height} Let $R$ be a commutative Noetherian ring of finite dimension $d$ such that all residue fields $R/\m$, for $\m \in \Max(R)$, are infinite. Let $R[x_1,\ldots,x_n]$ be a polynomial algebra over $R$, and let $J \subseteq R[x_1,\ldots,x_n]$ be an ideal with $\height(J) \gs d+1$. For all integers $0 \ls \ell \ls n$, there exist $r_1,\ldots,r_\ell \in R$ such that either $J + (x_1-r_1,\ldots,x_\ell-r_\ell) = R[x_1,\ldots,x_n]$, or $\height(J + (x_1-r_1,\ldots,x_\ell-r_\ell)) \gs d+\ell +1$ (when $\ell = 0$, the reader should think of $J+(x_1-r_1,\ldots,x_\ell-r_\ell)$ as being equal to $J$).
\end{Lemma}
\begin{proof} We proceed by induction on $\ell$, where the case $\ell=0$ follows from our assumptions. Let $1 \ls \ell \ls n$. By inductive hypothesis, we can find $r_1,\ldots,r_{\ell-1}$ such that either $J+(x_1-r_1,\ldots,x_{\ell-1}-r_{\ell-1}) = R[x_1,\ldots,x_n]$, or $\height(J+(x_1-r_1,\ldots,x_{\ell-1}-r_{\ell-1})) \gs d+\ell$. In the first case, any choice of $r_\ell$ will yield $J+(x_1-r_1,\ldots,x_\ell-r_\ell) = R[x_1,\ldots,x_n]$. In the second case, we have that $J+(x_1-r_1,\ldots,x_{\ell-1}-r_{\ell-1})$ is a proper ideal of height at least $d+\ell$. Let $\p_1,\ldots,\p_t$ be the minimal primes over $J+(x_1-r_1,\ldots,x_{\ell-1}-r_{\ell-1})$. For each $1 \ls j \ls t$, let $\q_j = \p_j \cap R$, and let $H_j := \{r \in R \mid x_\ell - r \in \p_j\}$. We want to show that $\bigcup_{j=1}^t H_j \ne R$. For this matter, from now on we can deal only with those indices $1\ls j \ls t$ for which $H_j \ne \emptyset$. In this case, we see that $H_j = a_j + \q_j$, i.e., $H_j$ is a coset determined by some $a_j \in R$. For each $\q_j$, choose a maximal ideal $\m_j$ of $R$ such that $\m_j$ contains $\q_j$. We then have a finite set $Y = \{\m_j \in \Max(R) \mid H_j \ne \emptyset\}$ of maximal ideals of $R$. For each $\m \in Y$, since $R/\m$ is infinite, there exists $b_\m \in R$ such that $b_\m + \m \notin \{a_j + \m_j \mid \m_j = \m, 1 \ls j \ls t\}$. By the Chinese Remainder Theorem, there exists $r_\ell \in R$ such that $r_\ell + \m = b_\m + \m$ for all $\m \in Y$. In particular, $r_\ell$ avoids all cosets
$H_j$, for $1 \ls j \ls t$. By definition of $H_j$, it follows that $x_\ell-r_\ell \notin \bigcup_{j=1}^t \p_j$, and this implies either that $J+(x_1-r_1,\ldots,x_\ell-r_\ell) = R[x_1,\ldots,x_n]$, or that $\height(J+(x_1-r_1,\ldots,x_\ell-r_\ell)) \gs d+\ell+1$, as desired.
\end{proof}

\begin{Remark} In Lemma~\ref{Lemma increase height}, if $R$ contains an infinite field $k$, we can actually choose $r_1, \ldots , r_\ell$ to be any generic elements inside $k$. In fact, given any proper ideal $I$ of $R[x_1,\ldots,x_n]$ and any fixed $f \in R[x_1,\ldots,x_n]$, no two distinct $r, s \in k \subseteq R$ can be such that both $f-r \in I$ and $f - s \in I$, otherwise the unit $r-s$ would be inside $I$. Therefore, with the same notation as in the inductive step in the proof of the Lemma, a generic choice $r_\ell \in k \subseteq R$ is such that $x_\ell-r_\ell$ avoids all the minimal primes $\p_1,\ldots,\p_t$ over $J+(x_1-r_1,\ldots,x_{\ell-1}-r_{\ell-1})$. 
\end{Remark}

We are now ready to present the main result of this section: Theorem~\ref{Serre's Theorem Intro} in the case when $R$ has infinite residue fields. As already pointed out in the introduction, we will prove Theorem~\ref{Serre's Theorem Intro} in its full generality in Section~\ref{Free-Basic Section}, using different methods.

\begin{Theorem} \label{Serre for infinite fields}
Let $R$ be a commutative Noetherian ring of dimension $d$ such that all residue fields $R/\m$, where $\m \in \Max(R)$, are infinite. Let $M$ be a finitely generated $R$-module such that $M_\p$ has a free summand of rank at least $d+1$ locally at every $\p \in \Spec(R)$. Then $M$ has a free summand.
\end{Theorem}
\begin{proof} Suppose that $M$ can be generated by $n$ elements, and fix a surjective map $\pi: R^n \to M$. Our assumptions about local number of free summands of $M$ guarantee that we can cover $\Spec(R)$ with open sets $D(s_1), \ldots, D(s_t)$, such that for all $1 \ls i \ls t$ the module $M_{s_i}$ has an $R_{s_i}$-free summand of rank at least $d+1$. In particular, for all $i$, we can find maps $M_{s_i} \to R^{d+1}_{s_i}$ such that the compositions $R^n_{s_i} \to M_{s_i} \to R^{d+1}_{s_i}$ are split surjections. Let $Q_i$ be the kernel of $R^n_{s_i} \to R^{d+1}_{s_i} \to 0$, so that we have a surjection $R_{s_i}[x_1,\ldots,x_n] \to \Sym_{R_{s_i}}(Q_i^*) \to 0$ at the level of symmetric algebras. For each $i$, consider the composition $R[x_1,\ldots,x_n] \to R_{s_i}[x_1,\ldots,x_n] \to \Sym_{R_{s_i}}(Q_i^*)$, and let $\a_i \subseteq R[x_1,\ldots,x_n]$ be its kernel. We claim that $\a_i$ is an ideal of height at least $d+1$. In fact, let $\Sym_{R_{s_i}}(Q_i^*) \cong R_{s_i}[x_1,\ldots,x_n]/J_i$ be a presentation of the symmetric algebra of $Q_i^*$ over $R_{s_i}$. Then $J_i$ can be generated by $d+1$ linear forms that, possibly after a change of variables, can be regarded as $d+1$ distinct variables. Therefore $J_i$ has height $d+1$ in $R_{s_i}[x_1,\ldots,x_n]$.
Let $I_i \subseteq R[x_1,\ldots,x_n]$ be any ideal such that $(I_i)_{s_i} = J_i$. Note that the ideal $\a_i$ is just the saturation $I_i : s_i^{\infty} = \{f \in R[x_1,\ldots,x_n] \mid s_i^mf \in I_i \mbox{ for some } m \in \N\}$,
and this implies that $\height(\a_i) \gs \height(J_i) = d+1$, since no minimal primes of $\a_i$ contain $s_i$.

Now consider the intersection $\a =\a_1 \cap \ldots \cap \a_n$, which is then an ideal of height at least $d+1$ inside $R[x_1,\ldots,x_n]$. Applying Lemma~\ref{Lemma increase height} to the ideal $\a$ with $\ell=n$, we obtain elements $r_1,\ldots,r_n \in R$ such that either $\a+(x_1-r_1,\ldots,x_n-r_n) = R[x_1,\ldots,x_n]$, or $\height(\a+(x_1-r_1,\ldots,x_n-r_n)) \gs d+n+1$. Since $\dim(R[x_1,\ldots,x_n]) = d+n$, the latter cannot happen. This implies that the ideals $\a_i+(x_1-r_1,\ldots,x_n-r_n)$ cannot be contained in $(\m,x_1-r_1,\ldots,x_n-r_n)$, for any $1 \ls i \ls t$ and any $\m \in \Max(R) \cap D(s_i)$. 
As a consequence, $(\m,x_1-r_1,\ldots,x_n-r_n)$ is not a maximal ideal of $\Sym_{R_{s_i}}(Q_i^*)$ for any $\m \in \Max(R) \cap D(s_i)$, for $1\ls i \ls t$. Now, Lemma~\ref{Lemma Max(Sym) VS MinGens} implies that the the image of $(r_1,\ldots,r_n)$ under $R^n \to M$ generates a non-trivial free $R_{s_i}$-summand locally on $D(s_i)$, for each $i$. Since the open sets $D(s_i)$ cover $\Spec(R)$, it generates a global free summand of $M$.
\end{proof}

\section{Free-basic elements}\label{Free-Basic Section}

Throughout this section, $R$ denotes a commutative Noetherian ring with identity. The purpose of this section is to prove Theorem~\ref{Serre's Theorem Intro} and Theorem~\ref{Bass's Theorem Intro} in full generality. We first recall the notion of a basic set. We refer the reader to \cite{EvansGriffith} for a more general and detailed treatment of basic sets and basic elements. We also want to mention \cite{HunekeNotesCA3} as a possible source for references.
\begin{Definition} A subset $X\subseteq \Spec(R)$ is said to be basic if it is closed under intersections. In other words, for any indexing set $\Lambda$ and any family $\{\p_\alpha\}_{\alpha \in \Lambda} \subseteq X$, whenever the intersection $I := \bigcap_{\alpha\in\Lambda}\p_\alpha$ is a prime ideal, then $I$ is still an element of $X$.
\end{Definition} 
\begin{Examples} Let $R$ be a commutative Noetherian ring.
\begin{enumerate}
\item $\Spec(R)$ is trivially a basic set.
\item $\Min(R)$ is a basic set. More generally, for an integer $n \in \N$, the set $X^{(n)} = \{\p \in \Spec(R) \mid \height(\p) \ls n\}$ is basic.
\item The set $j$-$\Spec(R) = \{ \p \in \Spec(R) \mid \p$ is an intersection of maximal ideals of $R\}$ is basic.
\end{enumerate}
\end{Examples}
Below is a list of properties about basic sets, which we leave as an exercise to the reader.

\begin{Proposition}\label{Properties about Basic Sets} Let $R$ be a commutative Noetherian ring, and $X$ be a basic set.
\begin{enumerate}
\item If $Y\subseteq X$ is closed, then $Y$ is a basic set.
\item Every closed set $Y\subseteq X$ is a finite union of irreducible closed sets in $X$. 
\item If $Y\subseteq X$ is an irreducible closed set then $Y=V(\p)\cap X$ for some $\p\in X$, i.e., $Y$ has a generic point.
\end{enumerate}
\end{Proposition}

Let $X$ be a basic set, and $M$ be a finitely generated $R$-module. An element $x \in M$ is called $X$-basic if $x \notin \p M_\p$, for all $\p\in X$. Equivalently, $x$ is a minimal generator of $M_\p$, for all $\p \in X$. The notion of basic element was introduced by Swan in \cite{Swan}, and later used by Eisenbud and Evans for proving Theorem~\ref{Serre's Theorem on Projective Modules} and Theorem~\ref{Bass's Cancellation Theorem}. When $P$ is a projective $R$-module, any $X$-basic element $x \in P$ produces a free summand $\lr{x}_\p$ of $M_\p$, for all $\p \in X$. In particular, when $X=\Spec(R)$ or just $X=j$-$\Spec(R)$, and $P$ is projective, any $X$-basic element produces a free summand $\lr{x}$ of $P$. In fact, in such case, the inclusion $\lr{x} \subseteq P$ splits for all $\m \in \Max(R) \subseteq X$, hence it splits globally. When $M$ is not necessarily projective, an $X$-basic element $x \in M$ may not generate a free summand, in general. In what follows, given an $R$-module $M$ and an element $x \in M$, we will denote $\lr{x}:= (R \cdot x)$. Similarly, given a subset $S \subseteq M$, we will denote $\lr{S}:=(R \cdot S)$. We now introduce an invariant that keeps track of the size of the local free summands of a module, rather than its local number of generators, as in the theory of basic elements. Some of the arguments that we present have a similar flavor as the treatment given for basic elements in \cite{HunekeNotesCA3}. 

\begin{Definition} \label{Definition delta}
Let $R$ be a commutative Noetherian ring, $M$ a finitely generated $R$-module, and $S$ a subset of $M$. For a prime $\p\in\Spec(R)$, we denote by $\delta_\p(S,M)$ the largest integer $n$ satisfying one of the following equivalent conditions:

\begin{enumerate}
\item There exists an $R_\p$-free direct summand $F$ of $M_\p$ such that the image of $S$ under the natural map $M_\p\rightarrow F\rightarrow F/\p F$ generates a $R_\p/\p R_\p$-vector subspace of rank $n$.
\item There exists an $R_\p$-free submodule $G$ of $\lr{S}_\p$ of rank $n$ such that $G$ splits off $M_\p$, i.e., there exists a surjective map $M_\p\rightarrow G$ such that the composition $G\subseteq \lr{S}_\p\subseteq M_\p\rightarrow G$ is the identity.
\end{enumerate} 
When the module $M$ is clear from the context, we will write $\delta_\p(S)$ for $\delta_\p(S,M)$.
\end{Definition}

Let $\p \in \Spec(R)$ be a prime, $M$ a finitely generated $R$-module, and $x \in M$. It follows from our definition that $\delta_\p(\{x\}) = 1$ if and only if $\lr{x}_\p \cong R_\p$, and the natural inclusion $\lr{x}_\p \subseteq M_\p$ splits. Also, observe that if $\q \subseteq \p$, then $\delta_\q(S) \gs \delta_\p(S)$. In fact, if $G_\p \subseteq \lr{S}_\p$ is a free summand of $M_\p$ of rank $\delta_\p(S)$, then $G_\q \subseteq \lr{S}_\q$ is a free summand of $M_\q$ of rank $\delta_\p(S)$. Since $\delta_\q(S)$ is defined as the maximum rank of any such modules, we have $\delta_\q(S) \gs \delta_\p(S)$.
\begin{Lemma}\label{Closed Set} Let $R$ be a commutative Noetherian ring, $M$ a finitely generated $R$-module, and $X$ a basic set. For any subset $S$ of $M$ and any integer $t\in\N$, the set $Y_t:=\{\p\in X\mid \delta_\p(S)\ls t\}$ is closed.
\end{Lemma}

\begin{proof} Observe that $Y_t=\{\p\in \Spec(R)\mid \delta_\p(S)\ls t\}\cap X$, so it is enough to show that $\{\p\in \Spec(R)\mid \delta_\p(S)\ls t\}$ is a closed set, i.e., that $\{\p\in \Spec(R)\mid \delta_\p(S)> t\}$ is open. Let $\p \in \Spec(R)$ be such that $\delta_\p(S)> t$. By definition of $\delta_\p(S)$, the identity map $R^{\delta_\p(S)}_\p\rightarrow R^{\delta_\p(S)}_\p$ factors as  $R^{\delta_\p(S)}_\p\subseteq \lr{S}_\p\subseteq M_\p\rightarrow R^{\delta_\p(S)}_\p$. The inclusion $R^{\delta_\p(S)}_\p\subseteq \lr{S}_\p$ and the surjection $M_\p\rightarrow R^{\delta_\p(S)}_\p$ lift to maps $R^{\delta_\p(S)}\rightarrow \lr{S}$ and $M \rightarrow R^{\delta_\p(S)}$, respectively. Let $K$ be the kernel and $C$ be the cokernel of the composition $R^{\delta_\p(S)} \to \lr{S} \subseteq M \rightarrow R^{\delta_\p(S)}$. As $K_\p=C_\p=0$ and both these modules are finitely generated, there is an element $s\in R\smallsetminus \p$ such that $K_\q=C_\q$ for all $\q\in D(s)$. Thus $D(s)$ is an open neighborhood of $\p$ such that $\delta_\q(S)\gs \delta_\p(S)>t$ for all $\q\in D(s)$, which shows that $\{\p\in \Spec(R)\mid \delta_\p(S)> t\}$ is indeed open. \end{proof}

\begin{Lemma}\label{Crucial Lemma} Let $R$ be a commutative Noetherian ring, $M$ a finitely generated $R$-module, and $X$ a basic set. For any subset $S$ of $M$ there exists a finite set of primes $\Lambda\subseteq X$ such that if $\p\in X \smallsetminus \Lambda$, there exists $\q\in \Lambda$ such that $\q\subsetneq \p$, and $\delta_\p(S)=\delta_\q(S)$.
\end{Lemma}

\begin{proof} For each $t\in\N$ the sets $Y_t:=\{\p\in X\mid \delta_\p(S)\ls t\}$ are closed by Lemma~\ref{Closed Set}. Also, observe that $Y_t =X$ for all $t\gg 0$. The sets $Y_t$ are finite unions of irreducible closed subsets by Proposition~\ref{Properties about Basic Sets}, and irreducible closed sets of $X$ are of the form $V(\p)\cap X$ for some $\p\in X$. Let $\Lambda$ be the finite collection of generic points of the finitely many irreducible components of the finite collection of closed sets $Y_t$, as $t$ varies through $\N$. Let $\p\in X \smallsetminus \Lambda$ and let $t=\delta_\p(S)$, so that $\p\in Y_t$ by definition. Let $\q$ be the generic point of an irreducible component of $Y_t$ which contains $\p$. Then $\q\subsetneq \p$ because $\q \in \Lambda$. Also, $t\gs \delta_\q(S)\gs \delta_\p(S)=t$, which implies the equality  $\delta_\q(S)= \delta_\p(S)$.  
\end{proof}

We now make the key definition of free-basic element.

\begin{Definition} Let $R$ be a commutative Noetherian ring, $M$ a finitely generated $R$-module, and $X$ a basic set. For a prime $\p \in X$, we say that $x\in M$ is a $\p$-free-basic element for $M$ if $\lr{x}_\p \cong R_\p$ is free, and the inclusion $\lr{x}_\p \subseteq M_\p$ splits. We say that $x$ is an $X$-free-basic element for $M$ if $x$ is a $\p$-free-basic element for all $\p\in X$. If the module $M$ is clear from the context, we will just say that $x$ is an $X$-free basic element.
\end{Definition}
Let $X$ be a basic set, and $\p \in X$ be a prime. Let $M$ be a finitely generated $R$-module, and $x \in M$. Observe that the element $x$ is $\p$-free-basic if and only if $\delta_\p(\{x\}) = 1$. More generally, for a prime $\p \in \Spec(R)$, we will say that a finite subset $S=\{x_1,\ldots,x_n\}$ of $M$ is a $\p$-free-basic set (for $M$) if $\delta_\p(S)\gs \min\{n, 1+\dim_X (\p)\}$. Here, $\dim_X (\p)$ denotes the supremum of the length of a chain $\p=\p_0\subsetneq \p_1\subsetneq\cdots \subsetneq \p_n$ with $\p_i\in X$ for each $i$. Note that, when $X=\Spec(R)$, we simply have $\dim_X(\p) = \dim(R/\p)$. We will say that $S$ is an $X$-free-basic set (for $M$) if $S$ is a $\p$-free-basic set (for $M$), for all $\p\in X$.

\begin{Lemma}\label{Main Lemma} Let $R$ be a commutative Noetherian ring, $M$ a finitely generated $R$-module, and $X$ a basic set. Let $S=\{x_1,\ldots,x_n\}\subseteq M$ be an $X$-free-basic set and $(a,x_1) \in R \oplus M$ be an $X$-free-basic element. Then there exist $a_1,\ldots,a_{n-1} \in R$ such that 
\[
\ds S'=\{x_1',x_2',\ldots,x_{n-1}'\}=\{x_1+aa_1x_n, x_2+a_2x_n,\ldots,x_{n-1}+a_{n-1}x_n\}
\] 
is $X$-free-basic, and $(a,x_1') \in R \oplus M$ is still an $X$-free-basic element.
\end{Lemma}

\begin{proof} Let $\Lambda$ be as in Lemma~\ref{Crucial Lemma}, we claim that for any choice of $a_1,a_2\ldots,a_{n-1}\in R$ and $\p\in X\smallsetminus \Lambda$ the set $\{x_1+aa_1x_n,x_2+a_2x_n,\ldots,x_{n-1}+a_{n-1}x_n\}$ is $\p$-free-basic. Indeed, let $a_1,\ldots,a_{n-1}\in R$ and $S'=\{x_1+aa_1x_n,\ldots,x_{n-1}+a_{n-1}x_n\}$. Let $F$ be a free $R_\p$-module such that the dimension of the subspace spanned by the image of the composition $\lr{S'\cup\{x_n\}}_\p=\lr{S}_\p\subseteq M_\p\rightarrow F\rightarrow F/\p F$ generates a $R_\p/\p R_\p$-vector subspace of rank $\delta_\p(S)$. Then the image of $\lr{S'}_\p$ under this composition is an $R_\p/\p R_\p$-vector subspace of dimension at least $\delta_\p(S)-1$. Therefore $\delta_\p(S')\gs \delta_\p(S)-1$. If $\p\in X \smallsetminus \Lambda$ then by Lemma~\ref{Crucial Lemma} there is a $\q\in\Lambda$ such that $\q\subsetneq \p$ and $\delta_\q(S)=\delta_\p(S)$. Hence $$\delta_\p(S')\gs \delta_\p(S)-1=\delta_\q(S)-1\gs \min\{n, 1+\dim_X(\q)\}-1\gs\min\{n-1, 1+\dim_X(\p)\}.$$

Suppose that $\Lambda = \{\q_1, \dotsc, \q_m\}$, and arrange the primes $\q_i$ in such a way that, for each $1 \ls \ell \ls m$, $\q_\ell$ is a minimal element of $\{\q_1, \dotsc, \q_l\}$. We prove, by induction on $\ell \gs 0$, that there exist $a_1,a_2,\ldots,a_{n-1}$ such that $\{x_1+aa_1x_n,x_2+a_2x_n,\ldots,x_{n-1}+a_{n-1}x_n\}$ is $\q_i$-free basic for all $1 \ls i \ls \ell$. Note that, for $\ell = m$, we get the desired claim for the lemma. When $\ell = 0$ we are done.

Note that, since $(a,x_1)$ is $X$-free-basic by assumption, any element of the form $(a,x_1+az)$, with $z \in M$, is $X$-free-basic as well. In fact, let $\p \in X$. If $a \notin \p$, then $(a,x_1+az) \cdot R_\p \subseteq (R \oplus M)_\p$ splits, since $a$ generates $R_\p$. If $a \in \p$, then $(a,x_1+az) \in (R \oplus M)_\p$ generates a free summand if and only if $(a,x_1)$ does, since $(a,x_1+az) \equiv (a,x_1)$ modulo $\p(R \oplus M)_\p$. But $(a,x_1)$ is $\p$-free basic by assumption. Either way, $(a,x_1+az)$ is $\p$-free-basic for all $\p \in X$. 

By induction on $\ell$, there exist $a_1,\ldots,a_{n-1}\in R$ such that the set $T:=\{x_1+aa_1x_n,x_2+a_2x_n,\ldots,x_{n-1}+a_{n-1}x_n\}$ is $\q$-free-basic for each $\q \in \{\q_1,\ldots,\q_{\ell-1}\}$. If $T$ happens to be $\q_\ell$-free-basic as well, we set $S'=T$, and we are done. Else, $\delta_{\q_\ell}(T)<\min\{n-1, 1+\dim_X (\q_\ell)\}$. Let $F$ be a free direct summand of $M_{\q_\ell}$ such that the $R_{\q_\ell}/{\q_\ell}R_{\q_\ell}$-vector space spanned by the image of $T$ under the natural map $M_{\q_\ell}\rightarrow F\rightarrow F/{\q_\ell}F$ has dimension $\delta_{\q_\ell}(T)$. We may assume, without loss of generality, that $F \subseteq \lr{T}_{\q_\ell}$. For each $x\in M$ let $\overline{x}$ denote the image of $x$ in $F/{\q_\ell}F$. The condition that $\delta_{\q_\ell}(T)< \min\{n-1,1+\dim_X (\q_{\ell})\}$ implies that the $R_{\q_\ell}/{\q_\ell}R_{\q_\ell}$-vector space spanned by $\{\overline{x_1+aa_1x_n},\ldots,\overline{x_{n-1}+a_{n-1}x_n}\}$ has dimension strictly smaller than $n-1$. Thus for some $1\ls i\ls n-1$, the $i$-th element in the above set is in the span of the the previous $i-1$ elements. We distinguish two cases.
\begin{itemize}
\item Assume $i\ne 1$. Let $r\in \left(\q_1\cap \cdots \cap \q_{\ell-1}\right) \smallsetminus \q_\ell$ and define
\[
\ds S':=\{x_1+aa_1x_n,\ldots, x_{i}+(a_i +r)x_n,\ldots,x_{n-1}+a_{n-1}x_n\}.
\] Since $r\in R \smallsetminus \q_\ell$, and because $\overline{x_i + a_ix_n}$ is in the span of the previous $i-1$ elements, the set $\{\overline{x_1+aa_1x_n},\ldots, \overline{x_{i}+(a_i +r)x_n},\ldots,\overline{x_{n-1}+a_{n-1}x_n}\}$ spans the same $R_{\q_\ell}/\q_\ell R_{\q_\ell}$-vector space as $\{\overline{x_1+aa_1x_n},\ldots, \overline{r x_n},\ldots,\overline{x_{n-1}+a_{n-1}x_n}\}$ which, in turn, spans the same  $R_{\q_\ell}/\q_\ell R_{\q_\ell}$-vector space as $\{\overline{x_1},\ldots,\overline{x_n}\}$. 

\item Now assume $i=1$, which means that $\overline{x_1+aa_1x_n} = 0$ in $F/\q_\ell F$. We claim that the element $a$ is not inside $\q_\ell$. In fact assume, by way of contradiction, that $a \in \q_\ell$. Then $(0,x_1+aa_1x_1)$ is $\q_\ell$-free-basic, because $(a,x_1)$ is $\q_\ell$-free-basic and $(a, x_1) \equiv (0,x_1+aa_1x_n)$ modulo $\q_\ell(R \oplus M)_{\q_\ell}$. Recall that $F \subseteq \lr{T}_{\q_\ell}$ is a free summand of $M_{\q_\ell}$ of rank $\delta_{\q_\ell}(T)$, and that the image of $T$ under $M_{\q_\ell} \to F \to F/\q_\ell F$ has $R_{\q_\ell}/\q_\ell R_{\q_\ell}$-vector space dimension equal to $\delta_{\q_\ell}(T)$. Set $x_1':=x_1+aa_1x_n \in T$, and let $\iota:F \to M_{\q_\ell}$ and $\pi:M_{\q_\ell} \to F$ denote the natural inclusion and projection, respectively. The fact that $\overline{x_1'} = 0$ simply means that $\pi(x_1') \in \q_\ell F$. If $x_1' \in F$, then $x_1' = \pi\iota(x_1') \in \q_{\ell}F$, so that $(0,x_1') \in \q_\ell(R \oplus M)_{\q_\ell}$, contradicting the fact that $(0,x_1')$ is $\q_\ell$-free-basic. Thus, we necessarily have $x_1' \notin F$. Let $y:=x_1'-\iota \pi(x_1')$, and note that $y \in \ker(\pi)$. Since $(0,x_1') - (0,y) = (0,\iota \pi(x_1')) \in \q_\ell(R \oplus M)_{\q_\ell}$, we have that $(0,y)$ is $\q_\ell$-free-basic as well. This means that $\lr{y}_{\q_\ell} \subseteq M_{\q_\ell}$ splits, and that $\lr{y}_{\q_\ell} \cong R_{\q_\ell}$. Because $M_{\q_\ell} \cong \ker(\pi) \oplus F$, and because $y \in \ker(\pi)$ splits out of $M_{\q_\ell}$, we have that $F' = \lr{y}_{\q_\ell} \oplus F$ is a free direct summand of $M_{\q_\ell}$. In addition, because $x_1' \in T$ and $\iota \pi(x_1') \in F \subseteq \lr{T}_{\q_\ell}$, we conclude that $y \in \lr{T}_{\q_\ell}$. Therefore $F' \subseteq \lr{T}_{\q_\ell}$ is a free direct summand of $M_{\q_\ell}$, and $F'$ has rank $\delta_{\q_\ell}(T)+1$, a contradiction. This shows that $a \notin \q_\ell$. Let $r \in \left(\q_1 \cap \ldots \cap \q_{\ell-1}\right) \smallsetminus \q_\ell$, and define
\[
\ds S':=\{x_1+a(a_1+r)x_n,x_2+a_2x_n,\ldots,x_{n-1}+a_{n-1}x_n\}.
\]
Since $ar \notin \q_\ell$ and $\overline{x_1+aa_1x_n} = 0$, the set 
\[
\ds \{\overline{x_1+a(a_1+r)x_n},\ldots, \overline{x_{i}+a_ix_n},\ldots,\overline{x_{n-1}+a_{n-1}x_n}\}
\]
spans the same $R_{\q_\ell}/\q_\ell R_{\q_\ell}$-vector space as $\{\overline{x_1},\ldots,\overline{x_n}\}$ inside $F/\q_\ell F$. 
\end{itemize}
Either way, we obtain that $\delta_{\q_\ell}(S')\gs \delta_{\q_\ell}(S)$. Since $S$ is $\q_\ell$-free-basic, we have that $\delta_{\q_\ell}(S) \gs \min\{n,1+\dim_X(\q_\ell)\}\gs \min\{n-1,1+\dim_X(\q_\ell)\}$. Moreover, $r\in \q_i$ for each $1\ls i\ls \ell-1$, therefore we also have that $\delta_{\q_i}(S')=\delta_{\q_i}(T) \gs \min\{n-1,1+\dim_X(\q_i)\}$ for each $1\ls i\ls \ell-1$. This completes the proof of the inductive step. As previously mentioned, the lemma now follows from choosing $\ell = m$.
\end{proof}

\begin{Theorem}\label{Free-basic elements exist} Let $R$ be a commutative Noetherian ring, $N \subseteq M$ finitely generated $R$-modules, and $X$ a basic set. Assume that, for each $\p\in X$, $N_\p$ contains a free $R_\p$-module of rank at least $1+\dim_X (\p)$ that splits off of $M_\p$. Then there exists $x\in N$ that is $X$-free-basic for $M$. Moreover, if an element $(a,y) \in R \oplus N$ is $X$-free-basic for $R\oplus M$, then $x$ can be chosen to be of the form $x=y+az$ for some $z \in N$. 

\end{Theorem}

\begin{proof} Assume that $(a,y) \in R \oplus N$ is $X$-free-basic for $R \oplus M$. Complete $y$ to a generating set $S=\{y,y_1,\ldots,y_n\}$ for $N$, which is easily verified to be $X$-free-basic for $M$, given our assumptions. Continued use of Lemma~\ref{Main Lemma} implies that there exist $a_1,\ldots,a_n \in R$ such that $x=y+a(a_1y_1+\ldots+a_ny_n)$ is $X$-free-basic for $M$, proving the last claim. As for the first claim, notice that $(1, y) \in R\oplus N$ is always $X$-free-basic for $R\oplus M$, for any $y \in N$. Thus existence follows from the last claim.
\end{proof}

\begin{Lemma} \label{delta=length} Let $(R,\m)$ be a local ring, $M$ a finitely generated $R$-module, and $S \subseteq M$ a subset. Consider $I(S,M):=\{y \in \lr{S} \mid f(y) \in \m$ for all $f \in \Hom_R(M,R)\}$ which is an $R$-submodule of $\lr{S}$. If $\lambda_R(-)$ denotes the length as an $R$-module, then
\[
\ds \delta_\m(S,M) = \lambda_R\left(\frac{\lr{S}}{I(S,M)}\right).
\]
\end{Lemma}
\begin{proof}
Let $\delta:=\delta_\m(S,M)$, and $F \subseteq \lr{S}$ be such that $F \cong R^\delta$ and $M = F \oplus T$ for some $R$-module $T \subseteq M$. By a slight modification of the argument in \cite[Discussion 6.7]{Huneke_HK_F-sig}, we see that $I(S,M) = \m F \oplus (T \cap \lr{S}) \cong \m^\delta \oplus (T \cap \lr{S})$. Since $\lr{S} = (F \oplus T) \cap \lr{S} = F \oplus (T \cap \lr{S}) \cong R^\delta \oplus (T \cap \lr{S})$, we finally have 
\[
\ds \delta = \lambda_R\left(\frac{R^\delta \oplus (T \cap \lr{S})}{\m^\delta \oplus (T \cap \lr{S})}\right) = \lambda_R\left(\frac{F \oplus (T \cap \lr{S})}{\m F \oplus (T \cap \lr{S})}\right) = \lambda_R\left(\frac{\lr{S}}{I(S,M)}\right). \qedhere
\]
\end{proof}
\begin{Lemma} \label{Induction delta} Let $R$ be a commutative Noetherian ring, $M$ a finitely generated $R$-module, $S$ a subset of $M$, and $F \subseteq \lr{S}$ a free $R$-module of rank $i$ such that $M = F \oplus M'$ for some $M' \subseteq M$. Let $S'$ be the projection of $S$ to $M'$ along the internal direct sum. For all $\p \in \Spec(R)$, we have $\delta_\p(S',M') = \delta_\p(S,M)-i$.
\end{Lemma}
\begin{proof}
After localizing at $\p$, we can assume that $(R,\m)$ is a local ring, with $\p=\m$.
Let $I(S,M)$ be as in Lemma~\ref{delta=length}, and similarly define $I(F,F)$ and $I(S',M')$. Given that $\lr{S} = F \oplus \lr{S'} \subseteq F \oplus M' = M$, it is straightforward to see that $I(S, M) = I(F, F) \oplus I(S', M')$, which gives rise the following (actually splitting) short exact sequence
\[
\xymatrixcolsep{5mm}
\xymatrixrowsep{2mm}
\xymatrix{
0 \ar[rr] && \ds  \frac{F}{I(F,F)} \ar[rr] && \ds \frac{\lr{S}}{I(S,M)} \ar[rr] && \ds \frac{\lr{S'}}{I(S',M')} \ar[rr] && 0.
}
\]
Since $F$ is a free $R$-module we have $I(F,F) = \m F$, because every minimal generator of $F$ is sent to a unit by some element of $\Hom_R(F,R)$. Therefore, Lemma~\ref{delta=length} and the short exact sequence above give
\[
\ds \delta_\m(S',M') = \lambda_R\left(\frac{\lr{S'}}{I(S',M')}\right) = \lambda_R\left(\frac{\lr{S}}{I(S,M)}\right)-\lambda_R\left(\frac{F}{\m F}\right) = \delta_\m(S,M) - i. \qedhere
\]
\end{proof}

Theorem~\ref{Serre's Theorem Intro} is a first consequence of the existence of free-basic elements for modules that, locally at every prime, have enough free summands. 
\begin{Theorem}\label{Generalized Serre} Let $R$ be a commutative Noetherian ring, $N \subseteq M$ finitely generated $R$-modules, $X= j$-$\Spec(R)$, and let $i$ be a positive integer. Assume that, for each $\p \in X$, $M_\p$ contains a free $R_\p$-summand $F(\p)$ of rank at least $i+\dim_X (\p)$. Then $M$ contains a free $R$-summand $F$ of rank at least $i$. In addition, if $F(\p) \subseteq N_\p$ for all $\p \in X$, then $F$ can be realized inside $N$.
\end{Theorem}

\begin{proof} It suffices to prove the statement involving $N$, since the first part of the theorem is nothing but the case $N=M$. We proceed by induction on $i \gs 1$. Let $i=1$, and assume that $F(\p) \subseteq N_\p$ is a free summand of $M_\p$ of rank at least $1+\dim_X(\p)$ for all $\p \in X$. Let $S$ be a generating set for $N$. Our assumptions guarantee that $\delta_\p(S,M) \gs 1+\dim_X(\p)$ for all $\p \in X$. Thus, there exists an element $x \in N$ that is $X$-free-basic for $M$, by Theorem~\ref{Free-basic elements exist}. In particular, this means that $\lr{x}_\m \cong R_\m$ and the inclusion $\lr{x}_\m \subseteq M_\m$ splits for each maximal ideal $\m \in \Max(R)$. Hence $F:=\lr{x}$ is a free $R$-submodule of $N$, and the inclusion $F \subseteq M$ splits globally. Now let $i>1$. By the base case $i=1$, there exists $x \in N$ such that $\lr{x} \cong R$ and $M = \lr{x} \oplus M'$ for some $R$-module $M' \subseteq M$. Let $S' = S \cap M'$. By Lemma~\ref{Induction delta}, we have that $\delta_\p(S',M') = \delta_\p(S,M)-1 \gs i-1+\dim_X(\p)$ for all $\p \in X$, and by induction we conclude that $M'$ has a free summand $F' \subseteq \lr{S'} \subseteq N$ of rank at least $i-1$. Therefore $F:= \lr{x} \oplus F'$ is a free submodule of $N$ of rank at least $i$, and it splits off of $M$.
\end{proof}

Note that Theorem~\ref{Generalized Serre} recovers the classical version of Serre's Splitting Theorem~\ref{Serre's Theorem on Projective Modules} \cite[Theorem 1]{Serre}.
Another consequence of Lemma~\ref{Main Lemma} is Theorem~\ref{Bass's Theorem Intro}. That is, a cancellation result for modules that, locally at each prime, have enough free summands. 

\begin{Theorem} \label{Bass's Cancellation general} Let $R$ be a commutative Noetherian ring, $M$ a finitely generated $R$-module, and $X=j$-$\Spec(R)$. Assume that, for each $\p \in X$, the module $M_\p$ contains a free $R_\p$-summand of rank at least $1+\dim_X (\p)$. Let $Q$ be a finitely generated projective $R$-module, and $N$ a finitely generated $R$-module such that $Q \oplus M \cong Q \oplus N$. Then $M \cong N$.
\end{Theorem}
\begin{proof}
Since $Q$ is projective, we can find another projective module $Q'$ such that $Q \oplus Q' \cong R^a$, for some integer $a$. We then obtain that $R^a \oplus M \cong R^a \oplus N$ and, by induction on $a$, we may assume that $R \oplus M \cong R \oplus N$. Let $\alpha:R\oplus N \to R \oplus M$ be an isomorphism, and set $\alpha((1,0)) = (a,x_1)$. We want to show that the composition
\[
\xymatrixcolsep{5mm}
\xymatrixrowsep{2mm}
\xymatrix{
R\oplus N \ar[rr]^-{\alpha} && R \oplus M \ar[rr]^-\beta && R \oplus M \ar[rr]^-\gamma && R \oplus M \ar[rr]^-\eta && R \oplus M \\
(1,0) \ar[rr] && (a,x_1) \ar[rrrr] &&&& (1,x) \ar[rr] && (1,0)
}
\]
is an isomorphism, with $\beta, \gamma$ and $\eta$ to be defined later. Note that, since $(1,0)$ is $X$-free-basic in $R \oplus M$, and $\alpha$ is an isomorphism by assumption, we have that $(a,x_1)$ is $X$-free-basic in $R \oplus M$. Let $\{x_1,\ldots,x_n\}$ be a set of generators of $M$, and recall that $\delta_\p(\{x_1,\ldots,x_n\}) \gs 1+\dim_X(\p)$ for all $\p \in X$, by assumption. By Theorem~\ref{Free-basic elements exist}, there exists an $X$-free-basic element $x \in M$ of the form $x=x_1+az$, for some $z \in M$. Define a map $\varphi:R \to M$, by setting $\varphi(1) = z$. Define $\beta:R \oplus M \to R \oplus M$ via the following matrix
\begin{eqnarray*}
\left[
\begin{array}{cc}
1_R & 0 \\ \varphi & 1_M
\end{array}
\right].
\end{eqnarray*}
It is easy to check that $\beta$ is an isomorphism. Note that $\beta((a,x_1)) = (a,\varphi(a)+x_1) = (a,az+x_1) = (a,x)$. Since $x \in M$ is $X$-free-basic, we have that $M \cong Rx \oplus M'$, for some $R$-module $M'$. Therefore, we may define an $R$-module map $\psi:M \to R$ as $\psi(x) = 1-a$. Define $\gamma: R \oplus M \to R \oplus M$ via the following matrix:
\begin{eqnarray*}
\left[
\begin{array}{cc}
1_R & \psi \\ 0 & 1_M
\end{array}
\right].
\end{eqnarray*}
The map $\gamma$ is an isomorphism, and it is such that $\gamma((a,x)) = (a+\psi(x),x) = (1,x)$. Finally, consider the map $\theta: R \to M$, defined as $\theta(1) = -x$. Define $\eta:R \oplus M \to R \oplus M$ via the matrix:
\begin{eqnarray*}
\left[
\begin{array}{cc}
1_R & 0 \\ \theta & 1_M
\end{array}
\right].
\end{eqnarray*}
The map $\eta$ is an isomorphism, and $\eta(1,x) = (1,x+\theta(1)) = (1,0)$. Consider the composition $\epsilon = \eta \circ \gamma \circ \beta \circ \alpha:R \oplus N \to R \oplus M$, which is an isomorphism such that $\epsilon((1,0)) = (1,0)$. We then have a commutative diagram
\[
\xymatrixcolsep{5mm}
\xymatrixrowsep{2mm}
\xymatrix{
0 \ar[rr] && R \ar[ddd]_{\cong}^-{id_R} \ar[rr] && R \oplus N \ar[rr]  \ar[ddd]_{\cong}^-{\epsilon} && N \ar@{-->}[ddd]\ar[rr] && 0 \\ \\ \\ 
0 \ar[rr] && R \ar[rr] && R \oplus M \ar[rr] && M \ar[rr] && 0
}
\]
that forces the induced map $N \to M$ to be an isomorphism.
\end{proof}

Bass's Cancellation Theorem for projective modules, Theorem~\ref{Bass's Cancellation Theorem} \cite[Theorem 9.1]{Bass}, now follows immediately from Theorem~\ref{Bass's Cancellation general}.

\section{Free summands with respect to $R$-submodules of \texorpdfstring{$\Hom_R(M,R)$}{Hom(M,R)}} \label{Cartier Section}

Throughout this section, $R$ denotes a commutative Noetherian ring with identity. Let $M$ be a finitely generated $R$-module. Theorem~\ref{Generalized Serre} provides a criterion for when $M$ has a free summand. Another way of saying this, it is a criterion that establishes when a surjective $R$-linear map $M\rightarrow R$ exists. The purpose of this section is to establish Theorem~\ref{Generalized Serre 2 Intro}, which is a generalization of Theorem~\ref{Serre's Theorem Intro}, and provides a criterion for when a surjective map $M \to R$ exists, and can be chosen to be inside a given $R$-submodule $\E$ of $\Hom_R(M,R)$. The authors use this results in \cite{DSPY} to establish the existence of a global F-signature of a Cartier subalgebra of $\C^R = \bigoplus_e \Hom_R(F^e_*R,R)$.

We now describe the setup and the notions that we will work with. 
Let $M$ be a finitely generated $R$-module, and $\E$ an $R$-submodule of $\Hom_R(M,R)$. We say that a summand $F$ of $M$ is a free $\E$-summand if $F\cong R^n$ is free and the projection $\varphi: M\to F\cong R^n$ is a direct sum of elements of $\E$. Observe that the choice of an isomorphism $F\cong R^n$ does not affect whether or not the projection $M \to F \cong R^n$ is a direct sum of elements of $\E$.
 
We now aim at defining free-basic elements with respect to $R$-submodules of $\Hom_R(M,R)$. Following the approach of Section~\ref{Free-Basic Section}, given a subset $S\subseteq M$ and a prime $\p\in\Spec(R)$, we denote by $\delta^{\E}_\p(S,M)$ the largest integer $n$ satisfying one of the following equivalent conditions:

\begin{enumerate}
\item There exists a free $\E_\p$-summand $F$ of $M_\p$ such that the image of $S$ under the natural map $M_\p\rightarrow F\rightarrow F/\p F$ generates a $R_\p/\p R_\p$-vector subspace of rank $n$.
\item There exists a free $\E_\p$-summand $G$ of $M_\p$ of rank $n$ which is contained in $\lr{S}_\p$.
\end{enumerate} 
Whenever $M$ is clear from the context, we denote $\delta_\p^\E(S,M)$ simply by $\delta_\p^\E(S)$.

\begin{Lemma}\label{Closed Set 2} Let $R$ be a commutative Noetherian ring, $M$ a finitely generated $R$-module, $X$ a basic set, and $\E$ an $R$-submodule of $\Hom_R(M,R)$. For any subset $S$ of $M$ and any integer $t\in\N$, the set $Y_t:=\{\p\in X\mid \delta^\E_\p(S)\ls t\}$ is closed.
\end{Lemma}

\begin{proof} As in the proof of Lemma~\ref{Closed Set}, it is enough to show $\{\p\in \Spec(R)\mid \delta^\E_\p(S)> t\}$ is open. Let $\p \in X$ be such that $\delta^\E_\p(S)> t$. By definition of $\delta_\p^\E(S)$, the identity map $R^{\delta_\p(S)}_\p \rightarrow R^{\delta_\p(S)}_\p$ factors as  $R^{\delta_\p(S)}_\p\subseteq \lr{S}_\p\subseteq M_\p\rightarrow R^{\delta^\E_\p(S)}_\p$ where $M_\p\rightarrow R^{\delta^\E_\p(S)}_\p$ is a direct sum of elements of $\E_\p$, i.e., $R^{\delta^\E_\p(S)}_\p$ is an $\E_\p$-summand. The map $M_\p\rightarrow R^{\delta^\E_\p(S)}_\p$ is the localization of a map $M\rightarrow  R^{\delta^\E_\p(S)}$, which is also a direct sum of elements in $\E$. One may now proceed as in Lemma~\ref{Closed Set}. 
\end{proof}

\begin{Lemma}\label{Crucial Lemma 2} Let $R$ be a commutative Noetherian ring, $M$ a finitely generated $R$-module, $X$ a basic set, $\E$ an $R$-submodule of $\Hom_R(M,R)$, and $S$ a subset of $M$. There exists a finite set of primes $\Lambda\subseteq X$ such that if $\p\in X \smallsetminus \Lambda$, there exists $\q \in \Lambda$ such that $\q\subsetneq \p$ and $\delta^\E_\p(S)=\delta^\E_\q(S)$.
\end{Lemma}

\begin{proof} Let $Y_t$ be as in Lemma~\ref{Closed Set 2}. Take $\Lambda$ to be the collection of the finitely many generic points of the finitely many closed sets $Y_t$ as $t$ varies in $\N$. The rest of the proof is identical to that of Lemma~\ref{Crucial Lemma}.\end{proof}

\begin{Definition}
Let $R$ be a commutative Noetherian ring, $M$ a finitely generated $R$-module, $X$ a basic set, and $\E$ an $R$-submodule of $\Hom_R(M,R)$. We say that $x\in M$ is a $\p$-free-basic element for $M$ with respect to $\E$ if $R_\p\cong \lr{x}_\p$ and the natural inclusion $\lr{x} \subseteq M$ splits via an element of $\E_\p$, when localized at $\p\in X$. We say that $x \in M$ is $X$-free-basic for $M$ with respect to $\E$ if $x$ is $\p$-free-basic for all $\p \in X$. Whenever no confusion may arise, we will just say that an element $x \in M$ is $X$-free-basic with respect to $\E$. 
\end{Definition}

Observe that, if $M$ is a finitely generated $R$-module, $X$ is a basic set and $\p \in X$ is a prime, then $x\in M$ is $\p$-free-basic with respect to $\E$ if and only if $\delta^\E_\p(\{x\})=1$. More generally, for $\p \in \Spec(R)$, we will say that a finite subset $S=\{x_1,\ldots,x_n\}$ of $M$ is a $\p$-free-basic set (for $M$) with respect to $\E$ if $\delta^\E_\p(S)\gs \min\{n, 1+\dim_X (\p)\}$. We will say that $S$ is an $X$-free-basic set (for $M$) with respect to $\E$ if $S$ is $\p$-free-basic with respect to $\E$ for all $\p\in X$. 

The proofs of the following two lemmas are similar to those of the corresponding results in Section~\ref{Free-Basic Section}. We include a sketch of the arguments for sake of completeness. Note that, given an $R$-submodule $\E$ of $\Hom_R(M,R)$, we can view $R \oplus \E \cong \Hom_R(R,R) \oplus \E$ as an $R$-submodule of $R \oplus \Hom_R(M,R) \cong \Hom_R(R \oplus M,R)$. 

\begin{Lemma}\label{Main Lemma 2} Let $R$ be a commutative Noetherian ring, $M$ a finitely generated $R$-module, $X$ a basic set, $\E$ an $R$-submodule of $\Hom_R(M,R)$. Let $S=\{x_1,\ldots,x_n\}\subseteq M$ be an $X$-free-basic set with respect to $\E$ and $(a,x_1)\in R\oplus M$ be an $X$-free-basic element with respect to $R \oplus \E$. There exist $a_1,\ldots,a_{n-1}$ such that 
\[
\ds S'=\{x_1',x_2', \ldots,x_{n-1}'\}=\{x_1+aa_1x_n,x_2+a_2x_n, \ldots,x_{n-1}+a_{n-1}x_n\}
\]
is $X$-free-basic with respect to $\E$, and $(a,x_1')$ is an $X$-free-basic element with respect to $R \oplus \E$.
\end{Lemma}

\begin{proof}
Let $\Lambda$ be as in Lemma~\ref{Crucial Lemma 2}. For any choice of $a_1,\ldots,a_{n-1}\in R$ and $\p\in X \smallsetminus \Lambda$, the set $\{x_1+aa_1x_n,\ldots,x_{n-1}+a_{n-1}x_n\}$ is $\p$-free-basic. In fact, it is enough to notice that, if we let $S'=\{x_1+aa_1x_n,\ldots,x_{n-1}+a_{n-1}x_n\}$, then $\delta^\E_\p(S')\gs \delta^\E_\p(S)-1$. If $\p\in X \smallsetminus \Lambda$ then by Lemma~\ref{Crucial Lemma 2} there exists $\q\in\Lambda$ such that $\q\subsetneq \p$ and $\delta^\E_\q(S)=\delta^\E_\p(S)$. Hence 
\[
\ds \delta^\E_\p(S')\gs \delta^\E_\p(S)-1=\delta^\E_\q(S)-1\gs \min\{n, 1+\dim_X(\q)\}-1\gs\min\{n-1, 1+\dim_X(\p)\}.
\]
The rest of the proof is done by induction, and it is along the same lines as the proof of Lemma~\ref{Main Lemma} in Section~\ref{Free-Basic Section}. The only main difference is that $\delta_\p(-)$ has to be replaced with $\delta_\p^\E(-)$, and all the maps $M_\p \to R_\p$ involved have to be elements of $\E_\p$. It is also worth noting that if $(a,x_1) \in R\oplus M$ is a $\p$-free-basic element with respect to $R \oplus \E$, then for any $z \in \p M$, the element $(a,x_1+z)$ is still $\p$-free-basic with respect to $R \oplus \E$.
\end{proof}

As an application of Lemma~\ref{Main Lemma 2}, we obtain the existence of free-basic elements with respect to $R$-submodules $\E$ of $\Hom_R(M,R)$.

\begin{Theorem}\label{Free-basic elements exist 2} Let $R$ be a commutative Noetherian ring, $N \subseteq M$ finitely generated $R$-modules, $X$ a basic set, and $\E$ an $R$-submodule of $\Hom_R(M,R)$. Assume that, for each $\p\in X$, $N_\p$ contains a free $R_\p$-submodule $F(\p)$ of rank at least $1+\dim_X (\p)$ that is a free $\E_\p$-summand of $M_\p$. Then there exists $x\in N$ that is $X$-free-basic for $M$ with respect to $\E$. Moreover, if an element $(a,y)\in R\oplus N$ is $X$-free-basic for $R \oplus M$ with respect to $R \oplus \E$, then $x$ can be chosen to be of the form $x=y+az$, for some $z \in N$.
\end{Theorem}
\begin{proof}
As for Theorem~\ref{Free-basic elements exist}, the last claim follows from a repeated use of Lemma~\ref{Main Lemma 2} on a generating set $S=\{y,y_1,\ldots,y_n\}$ of $N$. For existence, note that the element $(1,y) \in R \oplus N$ is $X$-free-basic for $R\oplus M$ with respect to $R \oplus \E$ for all $y \in N$. Then use the last claim.
\end{proof}

We now state two lemmas that will allow us to set up an induction process in Theorem~\ref{Generalized Serre 2}. We do not include proofs, as the arguments are completely analogous to those of Lemma~\ref{delta=length} and Lemma~\ref{Induction delta} from Section~\ref{Free-Basic Section}.
\begin{Lemma} \label{delta=length2} Let $(R,\m)$ be a local ring, $M$ a finitely generated $R$-module, $S$ a subset of $M$, and $\E$ an $R$-submodule of $\Hom_R(M,R)$. Consider $I^\E(S,M) := \{y \in \lr{S} \mid f(y) \in \m$ for all $f \in \E\}$, which is an $R$-submodule of $\lr{S}$. Then
\[
\ds \delta_\m^\E(S,M) = \lambda_R\left(\frac{\lr{S}}{I^\E(S,M)}\right).
\]
\end{Lemma}

\begin{Lemma}\label{Induction delta2} Let $R$ be a commutative Noetherian ring, $M$ a finitely generated $R$-module, $\E$ an $R$-submodule of $\Hom_R(M,R)$, and $S$ a subset of $M$. Suppose that $F \subseteq \lr{S}$ is a free $\E$-summand of $M$ of rank $i$, so that we have a split surjection $\varphi: M \to F \cong R^i$ which is a direct sum of elements in $\E$. Define $M' = \ker(\varphi)$, so that we can write $M = F \oplus M'$. Let $S'$ be the projection of $S$ to $M'$ along the internal direct sum, and $\E'$ be the projection of $\E$ to $\Hom_R(M',R)$. For all $\p \in \Spec(R)$, we have $\delta_\p^{\E'}(S',M') = \delta_\p^\E(S,M)-i$.
\end{Lemma}

Theorem~\ref{Generalized Serre 2 Intro} now follows from the existence of free-basic elements with respect to $\E$.
\begin{Theorem}\label{Generalized Serre 2} Let $R$ be a commutative Noetherian ring, $N \subseteq M$ finitely generated $R$-modules, $\E$ an $R$-submodule of $\Hom_R(M,R)$, and $X=j$-$\Spec(R)$. Assume that, for each $\p \in X$, $M_\p$ contains a free $\E_\p$-summand $F(\p)$ of rank at least $i+\dim_X(\p)$, for some positive integer $i$. Then $M$ contains a free $\E$-summand $F$ of rank at least $i$. Furthermore, if $F(\p) \subseteq N_\p$ for all $\p \in X$, then $F$ can be realized inside $N$.
\end{Theorem}
\begin{proof}
It suffices to prove the statement involving $N$. We proceed by induction on $i \gs 1$. If $i=1$, then our assumptions guarantee that there exists $x \in N$ that is $X$-free-basic for $M$ with respect to $\E$, by Theorem~\ref{Free-basic elements exist 2}. Thus, the inclusion $R_\m \cong \lr{x}_\m \subseteq M_\m$ splits via an element of $\E_\m$, for all maximal ideals $\m \in \Max(R)$. In particular, $F:=\lr{x}$ is a free $R$-submodule of $N$ such that the inclusion $F \subseteq M$ splits. We claim that the splitting map $M \to F \cong R$ can be chosen to be inside $\E$. In fact, consider the inclusion $R \cong \lr{x} \subseteq M$ and apply the functor $\Hom_R(-,R)$. We get an induced map $h:\E \subseteq \Hom_R(M,R) \to \Hom_R(R,R) \cong R$, and it suffices to show that $h$ is surjective. This follows from the fact that, for each $\m \in \Max(R)$, the map $h_\m:\E_\m \to R_\m$ is surjective, because the element $x$ is $\m$-free basic for $M$ with respect to $\E$. Now assume that $i>1$. By the base case $i=1$, there exists an element $x \in N$ with a splitting map $\varphi: M \to \lr{x} \cong R$ that belongs to $\E$. Define $M' = \ker(\varphi)$, so that $M= \lr{x} \oplus M'$. Let $S'$ be the projection of $S$ to $M'$ along the internal direct sum, and $\E'$ be the projection of $\E$ to $\Hom_R(M',R)$. By Lemma~\ref{Induction delta2} we have that $\delta_\p^{\E'}(S',M') = \delta_\p^\E(S,M)-1 \gs i-1 + \dim_X(\p)$, and by induction $M'$ has a free $\E'$-summand $F' \subseteq \lr{S'} \subseteq N$ of rank at least $i-1$. It follows that $F:=\lr{x} \oplus F'$ is a free $\E$-summand of $M$ of rank at least $i$, that is also contained in $N$.
\end{proof}

As a concluding remark we point out that, choosing $\E = \Hom_R(M,R)$, we recover the theory developed in Section~\ref{Free-Basic Section}. Therefore, the techniques introduced in Section~\ref{Cartier Section} are indeed a generalization of the theory developed in Section~\ref{Free-Basic Section}. 
\section*{Acknowledgments}
We thank Mohan Kumar for many helpful suggestions and ideas regarding the methods used in Section~\ref{SectionElemApproach}. We also thank Craig Huneke for very useful discussions.

\bibliographystyle{alpha}
\bibliography{References}

\end{document}